\documentclass[12pt,a4paper]{amsart}

\usepackage{latexsym}
\usepackage{amssymb}
\usepackage{amsthm}
\usepackage{amsfonts}
\usepackage{amsmath}
\usepackage{amstext}
\usepackage{amscd}
\usepackage[dvips]{graphics}
\usepackage{epsfig}
\usepackage{xcolor}

\numberwithin{equation}{section}

\theoremstyle{plain}
\newtheorem{thm}{Theorem}[section]

\newtheorem{cor}[thm]{Corollary}
\newtheorem{lem}[thm]{Lemma}

\newtheorem{theorem*}{Theorem}[]

\theoremstyle{definition}

\newtheorem{example}[thm]{Example}

\theoremstyle{remark}
\newtheorem{rem}[thm]{Remark}

\newtheorem{remark}{Remark}

\newcommand{\ds}{\displaystyle}
\newcommand{\R}{{\mathbb R}}

\newcommand{\N}{{\mathbb N}}

\newcommand{\vol}{{\rm Vol\,}}
\newcommand{\interior}{{\rm int\,}}

\newcommand{\ord}{{\rm ord}}

\title[Bounds for volumes of sub-level sets of polynomials and applications]
{Bounds for volumes of sub-level sets of polynomials and applications}

\author{ Ta L\^{e}  Loi and Minh Quy Pham}
\address[]{ Department of Mathematics, University of Dalat,
Dalat, Vietnam}
\email{loitl@dlu.edu.vn}
\email{p.minhquydl@gmail.com}

\subjclass{Primary 26D15, 42A38 Secondary 14P15, , 32B20 }

\keywords{Polynomial, sub-level set, oscillatory integral, singular integral.}

\begin{document}

\thanks{
This research is partially supported by Vietnamese National Foundation
for Science and Technology Development (NAFOSTED) – Project’s ID 101.04-2017.324.
}

\begin{abstract} In this paper, we present some explicit exponents in the estimates for the volumes of sub-level sets of polynomials on bounded sets, and applications to the decay of oscillatory integrals and the convergent of singular integrals.
\end{abstract}

\maketitle

\section{Introduction}
\medskip
{
{\color{black} Let $f$ be a real function on $\R^n$ and $A$ be a subset of $\R^n$.  We will denote by
 \[K_{A,f,t}=\{x\in A: \vert f(x)\vert\leq t\}\]
 the sub-level set of $f$, with $t\in \R_{\geq 0}$. We are interested in the  estimate for the volume of these sub-level sets:
\[V(t)= \vol(K_{A,f,t}), \text{ when } t\to 0.\]
}
This problem has been considered by many mathematicians and has a close relation to
the estimation of the decay of the oscillatory integral with phase $f$,
\[I(\lambda)=\left|\int_Ae^{i\lambda f(x)}dx\right|, \text{ when } \lambda \to \infty,\]
and the convergence of the singular integral
$\ds\int_A|f(x)|^{-\gamma}dx$. These objects are central in many field of physics and mathematics: Harmonis Analysis, Singularity theory,... We refer to some of them (\cite{St}, \cite{AGV},  \cite{Ro}, \cite{CCW}, \cite{CM},  \cite{DHSA}, \cite{Loi}, \cite{PSS}, ... )

{\color{black}
In this paper, we present some explicit exponents in the estimates for the volumes of sub-level sets of polynomials on bounded sets, and applications to the decay of oscillatory integrals and the convergent exponents of singular integrals. The methods using in this paper mainly come from geometry, which can be generalized to larger classes of sets and functions, such as definable sets in o-minimal structures, see \cite{Loi}. But the case of polynomial functions have its own interests, where strict upper bounds and lower bounds were unknown for general polynomial functions. We get these bounds using geometric methods, and give some explicit exponents which only depend on number of variables $n$, degree $d$, and the order $k'$ of $f$ at some point, see Theorem \ref{Th2}. Moreover, using these exponents from bounds for volumes of sub-level sets, we can determine the convergence of singular integrals of the form in Theorem \ref{Th3}.
}

More precisely, for $f:\R^n\to\R$ being a nonzero polynomial function of degree $d$, and $A$ a bounded domain in $\R^n$, we set
$d'\in\N\setminus\{0\}$ and $k'=\dim (\{x\in f^{-1}(0): \ord_x f\geq d'\}\cap\interior A)$.
Then we prove in Theorem \ref{Th1}, Theorem \ref{Th2}, and Theorem \ref{Th3} that there exist $C,C',C''>0$ such  that
\begin{itemize}
    \item[(1)] $C't^{\alpha'}\leq V(t) \leq Ct^{\alpha}$\ for all $t\geq 0$, where $\alpha=\frac{1}d$ and $\alpha'=\frac{n-k'}{d'}$, ($\alpha=\frac{n}{d}$ for $f$ satisfy condition given in Lemma \ref{Lemma3}),
    \item[(2)] $I(\lambda)\leq C''\lambda^{-\beta}$\ for all $\lambda>0$, where $\beta=\frac{1}{d}$,
    \item[(3)] $\int_A|f(x)|^{-\gamma}dx$ is convergent when $\gamma <\alpha$, and divergent when $\gamma\geq \alpha'.$
\end{itemize}

Now we explain the main ideas that leads to the above results.

To estimate the upper bound for the volume of sub-level set of $f$ (see Theorem \ref{Th1}), first we choose a good direction for the zero-set, i.e. a linear subspace $V$ such that
the projection $p_V|_{f^{-1}(0)}: f^{-1}(0)\to V^{\perp}$ is a finite map.
Then the restriction of the sub-level set on each fiber of $p_V$ has only isolated zeros,
from which  we establish the estimates for volumes for the sub-level sets of polynomials in one variable with coefficients depending on parameters, and we get the upper bounds for the volume.
 To get the lower bound, we use the homogeneous components to prove that  the sub-level sets
 always contain balls with suitable radius.

For oscillatory integrals, by applying Stationary Phase Principle, we treat the behavior of the phase function $f$ nearby their critical points. Similarly to the proof in \cite[Theorem 2]{Loi}, we partition the domain into two parts: on the first one  $\|\nabla f\|\leq t$, and on the second one $\|\nabla f\|> t$. Applying the estimate for volumes of sub-level sets for the first part, and van der Corput's lemma for the second part, then scaling
  $t$ with $\lambda$, we get the estimate for the decay (see Theorem \ref{Th2}).

For singular integrals, since $V(t)=\vol(K_{A,f,t})$ is nondecreasing in $t$, by the Stieltjes integration we have $\int_A|f|^{-\gamma}=\int_{\R}t^{-\gamma}dV(t)$. Using the integration by parts and the estimates for $V(t)$, we get the information of the convergence of the considered singular integrals (see Theorem \ref{Th3}).

Main theorems are Theorem \ref{Th1}  proved in Section 2, Theorem \ref{Th2} and Theorem \ref{Th3} proved in Section 3.

{\bf Notations}: In this paper, we will denote the Euclidean norm in $\R^n$ by
$\|\cdot\|$ ,  the $k$-dimensional ball with center $a$ and radius $r$ by $B^k(a,r)$,  the $n$-dimensional Lebesgue measure of $X\subset\R^n$ by $\vol_n(X)$ or $\vol(X)$, the gradient of $f$ by $\nabla f$ , the linear subspace generated by $v_1,\cdots, v_k\in\R^n$ by $L(v_1,\cdots,v_k)$.
}

\section{Upper and lower bounds for the volumes}
\medskip
In this section, we present some explicit exponents in the upper or  lower estimates for the volumes of sub-level sets of polynomial on bounded subset of $\R^n$.

First, for polynomials of one variable we have the following well-known  estimates.

\begin{lem}\label{Lemma1} Let $P(x)=a_dx^d+\cdots + a_0$ be a real polynomial and
$I$ be a bounded interval. Let $\lambda >0$. Then\\
{\rm (i)} If there exists $p\geq 1$ such that $|P^{(p)}(x)|\geq\lambda$, for all $x\in I$,
then
\[\vol_1(K_{I,P,t})\leq C_p \left(\ds\frac t{\lambda}\right)^{\frac 1p},\]
where $C_p$ is a constant depending only on $p$.
\\
{\rm (ii)} If $|a_0|-\sup_{x\in I}|P(x)-a_0|\geq \lambda$, then
\[\vol_1(K_{I,P,t})\leq C_0 \left(\ds\frac t{\lambda}\right)^{\frac 1d},\]
where $C_0=\vol_1(I)$.
\end{lem}
\begin{proof} For the proof of (i) see, for example, \cite[Prop. 2.1]{CCW}. To prove (ii), note that for $0\leq t <\lambda$, we have $|P(x)|\geq |a_0|-|P(x)-a_0|\geq\lambda >t$, and hence
$\vol_1(K_{I,P,t})=0$. Moreover, for all $t$, we have $\vol_1(K_{I,P,t})\leq \vol_1(I)$. Therefore,
$\vol_1(K_{I,P,t})\leq\frac{\vol(I)}{\lambda^{\frac 1d}}t^{\frac 1d}.$
\end{proof}
To apply the above lemma to estimate volumes  of sub-level sets of polynomials of $n$ variables, we need the following lemma.
 \begin{lem}\label{Lemma2} Let
$g(a,s)=a_0+a_1s+\cdots +a_{d}s^{d}$ be a polynomial of degree $d$,
where $a=(a_0,\cdots, a_d)$. For $m, R>0$,  set
\[K=\{a=(a_0,\cdots,a_d)\in\R^{d+1}:|a_0|+\cdots+|a_d|\geq m\} \text{ and } I=[-R,R].\]
Then there exists $\lambda=\lambda(d,m,R)>0$ such that for each $a\in K$,
 either there exists $p\in\{1,\cdots,d\}$ such that
\[\bigg|\ds\frac{\partial^p g}{\partial s^p}(a,s)\bigg|\geq \lambda, \text{ for all }  s\in I, \text{or } |a_0|-\sup_{s\in I}|g(a,s) -a_0|\geq\lambda\]
\end{lem}

\begin{proof}
First, we choose a family of positive numbers
$\varepsilon_d,\cdots, \varepsilon_0$
satisfying the following conditions is chosen:
\begin{itemize}
\item[(1)] $\varepsilon_p<\ds\frac{m}{d+1}$, for all $p\in\lbrace 0,\cdots,d\rbrace$.
\item[(2)] $p!\varepsilon_p>
    (p+1)!\varepsilon_{p+1}R+\cdots+\ds\frac{d!}{(d-p)!}\varepsilon_dR^{d-p}$,  for all $p\in\lbrace
 d-1,\cdots,0\rbrace$.
\end{itemize}
\noindent
For each $\varepsilon>0$, set
\begin{align*}
A_p(\varepsilon)=\lbrace a=(a_0,\cdots,a_d)\in K: \vert a_{p}\vert \leq \varepsilon\rbrace, p\in\{d,\cdots,0\}.
\end{align*}
For each  $a\in K$, two cases are to consider:
\\
{\it Case 1.d:} $a\not\in A_d(\varepsilon_d)$. Then
$\bigg\vert \ds\frac{\partial^{d}g}{\partial s^{d}}(a,s)\bigg\vert
=\vert d!a_{d}\vert\geq d!\varepsilon_{d}=\lambda_d$, for all $s\in I$.
\\
{\it Case 2.d:} $a \in A_d(\varepsilon_{d})$. In this case there are two subcases to consider:
\[a\in A_d(\varepsilon_{d})\setminus A_{d-1}(\varepsilon_{d-1}) \text{ or }
a\in A_d(\varepsilon_{d})\cap A_{d-1}(\varepsilon_{d-1}).\]
In general,  two cases to consider are:
\\
{\it Case 1.p:} $a\in
A_d(\varepsilon_{d})\cap\cdots\cap A_{p+1}(\varepsilon_{p+1})\setminus A_p(\varepsilon_{p})$.
\\
{\it Case 2.p:} $a\in A_d(\varepsilon_{d})\cap\cdots\cap A_{p+1}(\varepsilon_{p+1})\cap A_p(\varepsilon_{p})$.
\\
Nevertheless, by condition (1), there exists $q\in\{d,\cdots,0\}$ such that
$A_d(\varepsilon_{d})\cap A_{d-1}(\varepsilon_{d-1})\cap \cdots\cap A_{q}(\varepsilon_{q})=\emptyset$.
Henceforth,  there exists $p=p(a)\in\{d,\cdots,0\}$ such that $a$ falls in Case 1.p.
In this case, by the condition $(2)$,  we have
\begin{align*}
\bigg\vert \frac{\partial^{p}g}{\partial s^{p}}(a,s)\bigg\vert
&=
\bigg\vert p!a_p+ (p+1)!a_{p+1}s+\cdots+\frac{d!}{(d-p)!}a_{d}s^{d-p}\bigg\vert
\\
&\geq
p!\vert a_p\vert-
 \bigg\vert (p+1)!a_{p+1}s+\cdots + \frac{d!}{(d-p)!}a_ds^{d-p}\bigg\vert\\
	&\geq
p!\varepsilon_p-[(p+1)!\varepsilon_{p+1}R+\cdots+ \frac{d!}{(d-p)!}\varepsilon_dR^{d-p}]
=\lambda_p>0, \text{for all } s\in I.
 \end{align*}
Let $\lambda=\min\lbrace\lambda_d,\cdots,\lambda_0\rbrace$. Then, by the construction, to check that
$\lambda$ satisfies the demand of the lemma is easy.
\end{proof}

{\color{black}
\begin{lem}\label{Lemma3}
Let $f:\R^n\to \R$ be a polynomial of degree $d$. Assume that $f$ has isolated zero at $0$ and there is a constant $\epsilon_0>0$ such that for all $x\in B(0,\epsilon_0)\setminus \{0\}$
\begin{align*}
    \langle \nabla f(x), x\rangle\neq 0.
\end{align*}
Then for all  $u\in [0,1]$, $x\in B(0,\epsilon_0)$, we have
\begin{align*}
    \vert f(ux)\vert\leq \vert f(x)\vert.
\end{align*}
In particular, sub-level set $K_{B(0,\epsilon_0),f,t}$ is star-shaped with respect to center $0$, for all $u\geq 0$.
\end{lem}

\begin{proof}
Fix a point $x\in B(0,\epsilon_0)$ and set $g(u)=f(ux)$, $u\in [0,\infty]$.
Since $\langle \nabla f(y),y\rangle\neq 0$ with $y\in B(0,\epsilon_0)\setminus \{0\}$),  we have
\begin{align*}
    g'(u)=\langle \nabla f(ux),ux\rangle\neq 0, \forall u\in (0,1].
\end{align*}
Then $g'$ has constant sign on $(0,1]$, and hence,  $g$ is monotone on $[0,1]$. This implies
\begin{align*}
   (1)\hspace{1cm} \vert f(x')\vert =\vert g(u)\vert\leq \vert g(1)\vert =\vert f(x)\vert, \text{ for all } x'=ux \text{ with } u\in [0,1].
\end{align*}
The first part of the lemma is proved.

For the second part of the lemma, let  $x\in K_{B(0,\epsilon_0),f,t}$, and
 set  $[0,x]=\{x' :x'=ux, u\in [0,1]\}$.
Applying (1), we get
\begin{align*}
    \vert f(x')\vert \leq \vert f(x)\vert\leq t, \text{ for all } x'\in [0,x].
\end{align*}
 This implies that $x'\in K_{B(0,\epsilon_0),f,t}$, and hence,  $[0,x]\subset K_{B(0,\epsilon_0),f,t}$. This holds for all $x$ in $K_{B(0,\epsilon_0),f,t}$, so the sub-level set $K_{B(0,\epsilon_0),f,t}$ is star-shaped.
\end{proof}

\begin{example}
We give here a class of polynomial functions which satisfies the condition given in Lemma \ref{Lemma3}.
Let $f:\R^n\to \R$ be a homogeneous polynomial of degree $d$ and has isolated zero at $0$, by homogeneous Euler identity, we have
\begin{align*}
    \langle \nabla f(x),x\rangle =d\cdot f(x)\neq 0, \forall x\neq 0,
\end{align*}
which satisfy the condition in Lemma \ref{Lemma3}. Therefore one can conclude that the sub-level set $\{x\in B(0,R): \vert f(x)\vert\leq u\}$ is star-shaped for any $R,u\geq 0$.
\end{example}
}



Now we come to the main theorem of this part.

\begin{thm}\label{Th1}
 Let $f: \R^n\to\R$  be a nonzero polynomial  function of degree $d$ and $A$ be a bounded subset  of $\R^n$. Set
 \[V(t)= \vol(K_{A,f,t}), \text{and } Z_{d'}(f)=\{x\in f^{-1}(0): \ord_x f= d'\}
 \ (d'\in\N).\] Then the following assertions hold true:\\
{\rm (i)}  There exists $C=C(f,A)>0$ such that 
\[V(t)\leq Ct^{\frac{1}d},  \text{ for all } t\geq 0.\]
{\rm (ii)}   If $\dim \{Z_{d'}(f)\cap\interior A\}\geq k'$, then there exists $C'=C'(f,A,d',k')>0$ such that 
\[V(t)\geq C't^{\frac{n-k'}{d'}},  \text{ for all } t\geq 0.\]
{\color{black}
{\rm (iii)} If $0\in int(A)\neq \emptyset$ and $f$ satisfies the condition given in Lemma \ref{Lemma3}, then there exists $C'''(f,A)>0$ such that
\begin{align*}
    V(t)\leq C'''t^{\frac{n}{d}}, \text{ for all } t\geq 0.
\end{align*}
}

\end{thm}

\begin{proof}
(i) First, note that $V(t)\leq \vol(\{x\in \bar{A}: |f(x)|\leq t\})$, we can assume that $A=\bar{A}$, i.e. $A$ is compact.

If $f^{-1}(0)\cap A=\emptyset$, then by compactness $V(t)=0$ for all $t\geq 0$ sufficiently small, and the desired inequality follows.
\\
 Now, suppose that $f^{-1}(0)\cap A\neq \emptyset$, and hence $\dim f^{-1}(0)\geq 0$. Let $R=\sup\{\|x\|: x\in A\}$.

One can observe that $f^{-1}(0)$ is an algebraic set of dimension $\leq n-1$, so by \cite[Lemma 5.6]{P},
there exists an unit vector $e\in \R^n$ such that the following condition is satisfied:
\\
($*$) 
$\dim V_x\cap f^{-1}(0) \leq 0$, for all $x\in \R^n$,
where $V_x=x+\R e$.
\\
We set
$x=\hat{x}_e +x_e e, \text{where } x_e=\langle x,e\rangle$,
and $
 f(x)=f_e(\hat{x}_e,x_e)=a_{0}(\hat{x}_e)+\cdots+a_{d}(\hat{x}_e)x_e^{d}.
 $
 \\
Applying ($*$), one can imply that $f|_{V_{x}}$ has only isolated zeros for all $x\in\R^n$. This follows that $a_{0}, \cdots, a_{d}$ have no common zero.  Also, by the continuity of $a_{j}$'s and the compactness of $A$, there exists $m>0$, such that
\[m\leq |a_{0}(\hat{x}_e)|+\cdots+|a_{d}(\hat{x}_e)|, \text{for all } x\in A.\]
Now the conditions in Lemma \ref{Lemma2} are satisfy, so there exists $\lambda>0$ such that for each $x\in A$ either there exists $p\in\{1,\cdots, d\}$ such that
\[\bigg|\ds\frac{\partial^p f_e}{\partial x_e^p}(\hat{x}_e,x_e)\bigg|\geq \lambda, \text{for all }  x_e\in I,\]
or $|a_{0}(\hat{x}_e)|-\sup_{x_e\in I}|f_e(\hat{x}_e,x_e) -a_{0}(\hat{x_e})|\geq\lambda$.

Combine with parts (i) and (ii) of Lemma \ref{Lemma1}, we get $C'>0$, such that
\[\vol_1(\{x_e\in I: |f_e(\hat{x}_e,x_e)|\leq t\})
\leq C't^{\frac 1{d}}, \ \text{for all }  0\leq t\leq 1, x\in A.\]
Next, applying Fubini's theorem to estimate the volume of sub-level set, we have
\[\vol(K_{A,f,t})=\vol(\{x\in A:|f(x)|\leq t\})\leq C't^{\frac{1}d} \vol_{n-1}(p_e(A)) \leq C'' t^{\frac{1}{d}} \text{for all } 0\leq t\leq 1, \]
where $p_e:\R^n\to e^{\perp}$ is the orthogonal projection along $e$, and $C''=C'\vol_{n-1}(p_e(A))$.\\
Finally, taking $C=\max\{C'', {\vol(A)}\}$, we get
\[ V(t)=\vol(K_{A,f,t})\leq Ct^{\frac{1}d}, \text{for all } t\geq 0.\]

(ii) Since $Z_{d'}(f)$ is a semialgebraic set and, by the supposition,  $\dim Z_{d'}(f)\cap \interior A\geq k'$, there exists a manifold $\Gamma \subset Z_{d'}(f)\cap \interior A$ with $\dim \Gamma =k'$.  Moreover, $\Gamma$ can be chosen so that there exists a $(n-k')$-dimensional linear subspace $V$ of $\R^n$ such that the restriction to $\Gamma$ of the orthogonal projection along $V$, $p_V|_{\Gamma}:\Gamma\to p_V(\Gamma)\subset V^{\perp}$, is a bi-Lipschitz homeomorphism.
\\
Let $(e_1,\cdots,e_n)$ be a orthogonal basis of $\mathbb{R}^n$ such that $V=L(e_1,\cdots,e_{n-k'})$.
In this basis, we write
\begin{align*}
&x=(x_1,\cdots,x_n), \bar{x}=(x_1,\cdots,x_{n-k'})\in V,\hat{x}=(x_{n-k'+1},\cdots,x_n)\in V^{\perp}.
\end{align*}
For points $a\in \Gamma$ and  $x=(\bar{a}+\bar{u},\hat{a})\in V_a=a+V$, we set
\begin{align*}f|_{V_a}(x)=g_{\hat{a}}(\bar{u})=g_{\hat{a},d'}(\bar{u})+\cdots+g_{\hat{a},d}(\bar{u}),
\end{align*}
where $g_{\hat{a},j}$ is the homogeneous component of degree $j$ of $g_{\hat{a}}$.
\\
Then for all   $x=(\bar{a}+\bar{u},\hat{a})\in V_a=a+V$ with $0<\Vert \bar{u}\Vert<1$,
\[
\begin{array}{lll}
|f(x)|=|g_{\hat{a}}(\bar{u})| &\leq |g_{\hat{a},d'}(\bar{u})|+\cdots+|g_{\hat{a},d}(\bar{u})|\\
&\leq |g_{\hat{a},d'}(\frac{ \bar{u}}{\|\bar{u}\|})|\|\bar{u}\|^{d'}+\cdots+ |g_{\hat{a},d}(\frac{ \bar{u}}{\|\bar{u}\|})|\|\bar{u}\|^{d}\\
&\leq \max\{|g_{\hat{a},d'}(\bar{v})|+\cdots+|g_{\hat{a},d}({\bar{v}})|:\|\bar{v}\|=1\}\|\bar{u}\|^{d'}.
\end{array}
\]
By the continuity of $g_{\hat{a},j}$'s, shrinking $\Gamma$ if necessary, we have \[M=\sup\{|g_{\hat{a},d'}(\bar{v})|+\cdots+|g_{\hat{a},d}({\bar{v}})|:a\in\Gamma, \|\bar{v}\|=1\}<+\infty.\]
From the above estimate, for all $a\in \Gamma$ and $x=(\bar{a}+\bar{u},\hat{a})\in V_a$ with $\|\bar{u}\|<M^{-\frac 1{d'}}t^{\frac 1{d'}}<1$, we have
$|f(x)|<t$, i.e. $K_{A,f,t}\cap V_a $ contains a $(n-k')$-dimensional ball of radius $M^{-\frac 1{d'}}t^{\frac 1{d'}}$.
By the Fubini theorem or the coarea formula, see \cite[(3.2.22)]{F}, we get the lower bound for the volume of sub-level set
\[V(t)\geq  \vol_{k'}(p_V(\Gamma)) \vol_{n-k'}(B^{n-k'}(0, M^{-\frac 1{d'}}t^{\frac 1{d'}}))\geq \vol_{k'}(p_V(\Gamma))M^{-\frac{n-k'}{d'}}t^{\frac{n-k'}{d'}}, \]
for all $t>0$ sufficiently small.
 From this (ii) follows.

{\color{black}
(iii) 
Let $ \{e_1,\dots, e_n\}$ be a basis of $\R^n$.
 Since $f$ has isolated zero $0\in A$, from Lemma \ref{Lemma3}, the condition\\
\hspace{0.5cm}$(*)$ \ $\dim (V_x\cap f^{-1}(0))\leq 0$
\\
holds for all $x\in A$ and $V_x=x+\R e$, where $e\in \R^n\setminus\{0\}$ is any direction.
Then one can follow the proof in (i), for each $e\in S^{n-1}$, there exists $C_e'>0$, such that
\begin{align*}
    \vol_1(\{ x_e\in I: \vert f_e(\hat{x}_e,x_e)\vert\leq  t\})\leq C_e' t^{1/d}\text{ for all } 0\leq t\leq 1, x\in A.
\end{align*}
By Lemma \ref{Lemma3}, there exists $\epsilon_0>0$ such that the sub-level set
\begin{align*}
    K_{A\cap B(0,\epsilon_0),f,t}=\{ x\in A\cap B(0,\epsilon_0): \vert f(x)\vert\leq t\}
\end{align*}
is bounded and star-shaped. By using a suitable change of coordinate, we can choose the direction $e_1$ in the basis of $\R^n$ so that $C_{e_1}'\geq C_e $ for all $e\in \R^n$.
\\
If we put $t_0=\min\big\{1,\big(\frac{\epsilon_0}{C_{e_1}'}\big)^d\big\}$, then 
\begin{align*}
    K_{A\cap B(0,\epsilon_0),f,t}\subset B(0,C_{e_1}' t^{1/d})\subset B(0,\epsilon_0),\text{ for all } t\in [0,t_0].
\end{align*}
Therefore, for all $t\in (0,t_0]$, we have
\[\vol(K_{A\cap B(0,\epsilon_0),f,t})\leq \vol(B(0,C_{e_1}'t^{1/d}))\leq \vol(B(0,1))(C_{e_1}')^n t^{n/d}.\]
Finally, by the compactness of $A$, we see that $V(t)\leq\vol(A)\leq \frac{\vol(A)}{t_0^{n/d}}$ when $t_0\leq t$. So we can choose $C'''=\max\{\vol(B(0,1))(C_{e_1}')^n, \ds\frac{\vol(A)}{t_0^{n/d}}\}$ to get the upper bound 
\[ V(t)=\vol(\{x\in A:|f(x)|\leq t\})\leq C'''t^{\frac{n}d}, \text{for all } t\geq 0.\]
}

\end{proof}

{
\begin{cor} Let $f: \R^n\to\R$  be a nonzero polynomial  function of degree $d$ and $A$ be a bounded subset  of $\R^n$. 
 Then\\
{\rm (i)}  If $f^{-1}(0)\cap\interior A\neq\emptyset$, then
  there exist $C,C'>0$ such that 
\[C't^n\leq V(t)\leq Ct^{\frac{1}d},  \text{ for all } t\geq 0.\]
{\rm (ii)}   If there is $a\in f^{-1}(0)\cap\interior A$ with $\ord_af=d'$, then there exists $C''>0$ such that 
\[V(t)\geq C''t^{\frac{n}{d'}},  \text{ for all } t\geq 0.\]
\end{cor}
\begin{proof} Apply Theorem \ref{Th1} with $d'=1$ for (i), and $k'=0$ for (ii).

\end{proof}}

In the sequel, we will denote
\par
$f(t) \precsim g(t)$ iff there exists $C>0$ so that $ f(t)\leq Cg(t)$, for all $t>0$ sufficiently small.
\par
$f(t)\thickapprox g(t)$ iff $g(t)\precsim f(t)$ and $f(t)\precsim g(t)$.

\begin{remark} \label{rm1}
(i) The exponents in Theorem \ref{Th1} are explicit and easy to calculate  (cf. \cite[Theorem 7.1]{CCW} and \cite[Theorem 3.1,3.2]{DHSA}).
{\color{black}
}
\\
(ii) Due to the lower-bounds, in some cases, we can get the first exponent in the asymptotic expansion of $V(t)$, i.e. the exponent $\alpha_f$ so that $V(t)\thickapprox t^{\alpha_f}$ (see examples below).
{
\color{black} From Theorem \ref{Th1}, we have $\alpha_f\in [\ds\frac{1}d, \frac{n-k'}{d'}]$. In the case when $f$ satisfies the condition in Lemma \ref{Lemma3}, we get $\alpha_f\in [\ds\frac{n}d, \frac{n-k'}{d'}]$.
}
Nevertheless, in general, the data $(n, d=\deg f, k=\dim f^{-1}(0), d'=\ord_a f, k'=\dim Z_{d'}(f))$ are not sufficient  to get  $\alpha_f$.
For the case where $f$ is a homogeneous polynomial of even degree $d$ under the assumption that the sub-level set has finite volume on $\R^n$, it is  proved that  $\alpha_f=\frac{n}{d}$\
(see \cite{MS1},\cite{MS2} or  \cite[Theorem 2.2]{La}).
{\color{black}
}
\end{remark}

\begin{example}\label{Ex1}
{\color{black}
}
a) Let $f$ be a polynomial of degree $d$ in $n$ variables and $A$ be a bounded set. If $f$ has $\dim f^{-1}(0)=n-1$ and there exists $ a\in f^{-1}(0)$ such that $\ord_a f=d$, then $V(t)\thickapprox t^\frac{1}{d}$.
\\
b) For $f$ be a homogeneous polynomial of degree $d$ in $n$ variables, in general, we can not have $V(t)\thickapprox t^\frac{n-k}{d},\forall t\geq 0$, $k=\dim f^{-1}(0)$.
\\
For example, consider $f(x,y)=x^2y^2(x^2+y^2)$ on $A=[0,1]^2$. We have $f^{-1}(0)=Ox\cup Oy$ and $\dim f^{-1}(0)=1$, hence by Theorem \ref{Th1}, we get $C_1,C_2>0$:
$$C_1t^\frac{1}{2}\leq V(t)\leq C_2t^\frac{1}{6},\forall t\geq 0.$$
While by \cite[Theorem 2.2]{La}, we have $V(t)\thickapprox t^\frac{1}{3}$.
\end{example}

\section{Applications}
\medskip
In this section, applying Theorem \ref{Th1}, we give some explicit exponents in the estimates the decay of oscillatory integrals with phase functions being polynomials, and the convergent of integrals of the form $\int_A|f|^{-\gamma}$, where $f$ is a polynomial and $A$ is a bounded domain in $\R^n$.

First, we recall the van der Corput lemma.

\begin{lem}[van der Corput]\label{Corput}
Let $f:(a,b)\to\R$ be a $C^1$ function. Fix $t>0$.
Suppose that
$|f'(x)|\geq t, \forall x\in (a,b)$,
 and $f' $ is monotonic.
Then
\[ \left|\int_a^be^{i\lambda f(x)}dx\right|\leq 3(\lambda t)^{-1}, \textsl{ for all } \lambda>0.\]
\end{lem}

\begin{proof} See, for example, \cite[Ch. VIII, Proposition 2]{St} or \cite[Prop. 2.2]{CCW}.
\end{proof}

For $g:\R^n\to\R$ being $C^1$ function with compact support, set
\[ \|g\|_{\infty}=\sup_{\R^n}|g|, \text{ and } \|\nabla g\|_1=\int_{\R^n}\|\nabla g\|.\]

\begin{thm}\label{Th2}
 Let $f: \R^n\to\R$  be  a  polynomial function of degree $d\geq 1$. Then for any semialgebraic compact subset $A$ of $\R^n$, there exists $C=C(f,A)>0$ such that
for any $C^1$ function $g:\R^n\to\R$ with compact support containing in $A$, we have
\[ \left|\int_Ae^{i\lambda f(x)}g(x)dx\right|\leq C\lambda^{-\frac{1}{d}}(\|g\|_{\infty}+\|\nabla g\|_1), \text{ for all } \lambda> 0.\]
\end{thm}

\begin{proof} (cf. \cite{Loi}) For each $t>0$, put $A=A_t\cup B_t$,
 where
 \[A_t=\{x\in A: \|\nabla f(x)\|\leq t\}, \
 B_t=\{x\in A: \|\nabla f(x)\|>t\}.\]
 We will estimate the integrals on each set of the union.
\\
To estimate the integral on $A_t$, we apply Theorem \ref{Th1} to get $C_1>0$ so that
\[ \vol(A_t)= \vol(\{x\in A:\|\nabla f(x)\|^2\leq t^2\})\leq
C_1t^{\frac{1}{d-1}}.\]
Therefore,
\[\left|\ds\int_{A_t}e^{i\lambda f( x)}g(x)dx\right|
\leq \vol(A_t)\|g\|_{\infty}
\leq C_1t^{\frac{1}{d-1}} \|g\|_{\infty}.
\]
To estimate the integral on $B_t$, we apply the van der Corput Lemma \ref{Corput}. First,  note that
\[
B_t\subset  \bigcup_{k=1}^n \{x\in A: |\partial_kf(x)|\geq \textstyle\frac{t}{n} \}, 
\]
For $k=n$, let $x=(x',x_n)$ denote a point in $B_t\subset \R^{n-1}\times I$,  $\tilde{B}_t$ the projection of $B_t$ to the first $n-1$ coordinates, and $I$ denote an interval so that $B_t\subset\tilde{B}_t\times I$.
For each $x'\in\tilde{B}_t$, since $B_t$ is a semialgebraic set, there exists $N\in\N$, independing of $x'$ and $t$, such that the set
\[\{ x_n\in I:  |\partial_nf(x',x_n)|\geq \textstyle\frac{t}{n}\}\]
is the union of atmost $d$ intervals, say $I_1,\cdots, I_d$ (depending on $x'$ and $t$,  and some of them may be empty).
On each of the intervals, say $I_j=(a,b)$,
Let $F(x',x_n)=\int_a^{x_n}e^{i\lambda f(x',s)}ds, x_n\in I_j$.
Since $|\partial_nf(x',x_n)|\geq t/n$ on $I_j$, by the van der Corput Lemma \ref{Corput}, $|F(x',x_n)|\leq 3(\frac{\lambda t}n)^{-1}$.
Integrating by parts and using this inequality,  we get
\[\begin{array}{ll}
\left|\ds\int_{ I_j}e^{i\lambda f(x',x_n)}g(x',x_n)dx_n\right|
&=\left|\ds\int_{I_j}\partial_nF(x',x_n)g(x',x_n)dx_n\right|
\\
&\leq  3(\frac{\lambda t}n)^{-1}2\|g\|_{\infty}+ 3(\frac{\lambda t}n)^{-1}
\ds\int_{ I_j}|\partial_ng(x',x_n)|dx_n.
\end{array}
\]
 Applying the  Fubini Theorem and  the above estimation on each of the intervals, we get
\[\begin{array}{llll}
\left|\ds\int_{ |\partial_nf|\geq \frac{t}n}e^{i\lambda f(x)}g(x)\chi_{B_t}(x)dx\right|
&\leq
\ds\int_{\tilde{B}_t}\left(\sum_{j=1}^d\left|\int_{I_j}e^{i\lambda f(x',x_n)}g(x',x_n)dx_n\right|\right)dx' \\
&\leq
 \ds\int_{\tilde{B}_t}\left(\ds\sum_{j=1}^d \left(3(\textstyle\frac{\lambda t}n)^{-1}(2\|g\|_{\infty}+\int_{I_j}|\partial_ng(x',x_n)|dx_n)\right)\right)dx'\
\\
&\leq C_2(\lambda t)^{-1}(\|g\|_{\infty}+\|\nabla g\|_1),
\end{array}\]
where $C_2=\max_{1\leq k\leq n}\text{Vol}_{n-1}(p_k(A))d6n$, and
$p_k:\R^n\to\R^{n-1}$ is the projection missing the $k$-$th$ coordinate.
\\
 Using the similar estimations for $k=1,2,\cdots $, we get $C_3>0$ so that
\[
\begin{array}{ll}
\left|\ds\int_{B_{t}}e^{i\lambda f(x)}g(x)dx\right|
&\leq \ds\sum_{k=1}^n\left|\ds\int_{ |\partial_kf|\geq \frac tn}e^{i\lambda f(x)}g(x)\chi_{B_t}(x)dx\right|\\
&\leq C_3(\lambda t)^{-1}(\|g\|_{\infty}+\|\nabla g\|_1). 
\end{array}
\]
From the above estimates, we have
 \[
 \begin{array}{lll}
 \left|\ds\int_Ae^{i\lambda f(x)}g(x)dx\right|
& \leq \left|\int_{A_t}e^{i\lambda f( x)}g(x)dx\right|+\left|\int_{B_t } e^{i\lambda f(x)}g(x)dx\right|\\
& \leq (C_1t^{\frac{1}{d-1}}+C_3(\lambda t)^{-1})(\|g\|_{\infty}+\|\nabla g\|_1)
 \end{array}
 \]
 Now we choose $t=\lambda^{-\delta}$, so that $t^{\frac{1}{d-1}}=(\lambda t)^{-1}$,
 to get $\delta=\frac{d-1}{d}$, and hence
\[ \left|\int_Ae^{i\lambda f(x)}g(x)dx\right|\leq C\lambda^{-\frac{1}{d}}(\|g\|_{\infty}+\|\nabla g\|_1), \text{ for all } \lambda> 0,\]
where  $C=C_1+C_3$.
\end{proof}


\begin{thm}\label{Th3}
 Let $f: \R^n\to\R$  be a nonzero polynomial  function.
 Let $A$ be a bounded subset $\R^n$. Put $V(t)=\vol(\{x\in A: |f(x) |\leq t\})$.
 \\
 {\rm (i) } If $V(t)\leq Ct^{\alpha}$ for all $t\geq0$ sufficiently small, then
$ \int_A|f(x)|^{-\gamma}dx<+\infty$, when $\gamma<\alpha$. \\
{\rm (ii) } If $V(t)\geq C't^{\alpha'}$ for all $t\geq0$ sufficiently small, then
$ \int_A|f(x)|^{-\gamma}dx=\infty$, when $\gamma\geq \alpha'$. \\
{
 As a consequence, when $\deg f=d$, and $\dim Z_{d'}(f)\cap\interior A\geq k'$,
  one can choose $\alpha=\frac{1}d, \alpha'=\frac {n-k'}{d'}$.
{\color{black}
If $f$ satisfies the condition given in Lemma \ref{Lemma3}, then $\alpha=\frac{n}{d}$.
}
  }
\end{thm}
\begin{proof} Note that $V(t)= \vol(\{x\in A: |f(x)|\leq t\})$ is a nondecreasing function in $t$. Let $T=\sup_A |f|$.
By the change variables, see \cite[2.5.18 (3) and 2.4.18 (1)]{F},
\[\int_A|f(x)|^{-\gamma}dx =\int_0^T t^{-\gamma}dV(t).\]
Applying the integration by parts, see \cite[2.6.7]{F}, we have
\[\int_0^T t^{-\gamma}dV(t)=t^{-\gamma}V(t)|_0^T+\gamma\int_0^Tt^{-\gamma-1}V(t)dt.\]
(i) If $V(t)\leq Ct^{\alpha}$, then $0\leq t^{-\gamma}V(t)\leq Ct^{-\gamma+\alpha}$ and $\int_0^Tt^{-\gamma-1}V(t)dt\leq \int_0^TCt^{-\gamma-1+\alpha}dt$. Therefore,  $t^{-\gamma}V(t)|_0^T<+\infty$ and $\gamma\int_0^Tt^{-\gamma-1}V(t)dt<+\infty$ when $\gamma<\alpha$, and hence,
$\int_A|f(x)|^{-\gamma}dx<+\infty$, when $\gamma<\alpha$.\\
(ii) Similarly, if  $V(t)\geq C't^{\alpha'}$, then $ t^{-\gamma}V(t)\geq C't^{-\gamma+\alpha'}$ and $\int_0^Tt^{-\gamma-1}V(t)dt\geq \int_0^TC't^{-\gamma-1+\alpha'}dt$. Therefore,  $t^{-\gamma}V(t)|_0^T=+\infty$ and $\gamma\int_0^Tt^{-\gamma-1}V(t)dt=+\infty$ when $\gamma\geq\alpha'$, and hence,
the considered integral is divergent when $\gamma\geq \alpha'$.
\\
The last part of the theorem is followed from Theorem \ref{Th1}.
\end{proof}

\begin{rem} The {\it integration index } of $f$ on $A$, is defined by
\[ i(f,A)=\sup\{\gamma: \int_A|f(x)|^{-\gamma}dx<+\infty\}.\]
By the Theorem \ref{Th1} and Theorem \ref{Th3}, $i(f,A)\in [\frac{1}d,\frac {n-k'}{d'}]$, for any polynomial function $f$.
{\color{black} In particular, $i(f,A)\in [\frac{n}{d},\frac{n-k'}{d'}]$, when $f$ satisfies the condition in Lemma \ref{Lemma3}.
}
Therefore, when the endpoints are equal, $i(f,A)$ is determined. But there are cases
where $i(f,A)>\frac{1}d$
See the following examples.
\end{rem}

\begin{example}\label{Ex3}
a) From Theorem \ref{Th3} and Example \ref{Ex1}, if $f$ be a polynomial of degree $d$ in $n$ variables and $A$ be a bounded set, $\dim f^{-1}(0) = n-1$ and there exists $ a\in f^{-1}(0)$ such that $\ord_a f = d$,  then in a bounded neighborhood $A$ of the origin we have
$\int_A|f(x)|^{-\gamma}dx<+\infty$ if and only if $\gamma<i(f,A)=\frac{1}d$.
\\
{\color{black}
b) Let $g(x,y) = x^2+y^2$, and $A = \{ (x,y)\in \mathbb{R}^2: x^2+y^2\leq 1\}$.
Then we have $d=\deg g = 2$, $f^{-1}(0) = \{(0,0)\}$ has dimension $0$.
Then, by Theorem \ref{Th1}, (iii), we have
$$t=t^\frac{n-k'}{d'}\leq V(t) \leq t^\frac{n}{d}=t.$$
By the above remark, we have $i(g,A) =1$.
}
\\
c) In the above Theorem \ref{Th3}, if $f^{-1}(0)\cap A =\emptyset$ or $\forall d' \in \mathbb{N}, \dim Z_{d'}\cap\text{int} A = \emptyset$, then $i(f,A)$ can be equal $\infty$.
\end{example}

\end{document}